\documentclass{article}

\def\version{January 23, 2013}

\usepackage{upgreek}
\usepackage{latexsym}
\usepackage[dvips]{epsfig}
\usepackage{bm,amssymb}
\usepackage{amsthm,mathrsfs}
\usepackage{amsmath}
\usepackage{mathrsfs}
\usepackage{times}

\nonstopmode

\newcommand{\notyet}[1]{}

\DeclareSymbolFont{AMSb}{U}{msb}{m}{n}
\DeclareSymbolFontAlphabet{\mathbb}{AMSb}

\newcommand{\be}{\begin{equation}}
\newcommand{\ee}{\end{equation}}
\newcommand{\beqn}{\begin{eqnarray}}
\newcommand{\eeqn}{\end{eqnarray}}

\newcommand{\ba}{\begin{array}}
\newcommand{\ea}{\end{array}}

\newcommand{\supp}{\mathop{\rm supp}}

\newcommand{\at}[1]{\vert\sb{\sb{#1}}}
\newcommand{\At}[1]{\Big\vert\sb{\sb{#1}}}

\providecommand{\C}{\mathbb{C}}
\renewcommand{\C}{\mathbb{C}}
\newcommand{\R}{\mathbb{R}}
\newcommand{\N}{\mathbb{N}}
\newcommand{\Z}{\mathbb{Z}}
\newcommand{\T}{\mathbb{T}}

\newcommand{\abs}[1]{\vert #1 \vert}

\newcommand{\sothat}{{\rm :}\ }

\providecommand{\ltor}[1]{
\ifnum #1=1{\it i}\else\ifnum #1=2{\it ii}\else\ifnum #1=3{\it iii}
\else\ifnum #1=4 {\it iv}\fi\fi\fi\fi
}

\DeclareMathSymbol{\varGamma}{\mathord}{letters}{"00}
\DeclareMathSymbol{\varDelta}{\mathord}{letters}{"01}
\DeclareMathSymbol{\varSigma}{\mathord}{letters}{"06}
\DeclareMathSymbol{\varPhi}{\mathord}{letters}{"08}
\DeclareMathSymbol{\varOmega}{\mathord}{letters}{"0A}

\theoremstyle{plain}
\newtheorem{theorem}{Theorem}
\newtheorem{corollary}{Corollary}
\newtheorem{lemma}{Lemma}

\theoremstyle{definition}

\newtheorem{remark}{Remark}

\makeatletter\@addtoreset{equation}{section}
\makeatother

\begin{document}

\title{
On the Titchmarsh convolution theorem
for distributions on the circle}

\author{
\\
{\sc Andrew Comech}
\\
{\it\small
Texas A\&M University,
College Station, TX 77843, USA}
\\
{\it\small
Institute for Information Transmission Problems,
Moscow 101447, Russia}
\\
{\sc Alexander Komech}
\\
{\it\small
University of Vienna, Wien A-1090, 
Austria}
\\
{\it\small
Institute for Information Transmission Problems,
Moscow 101447, Russia}
}

\date{\version}

\maketitle

\begin{abstract}
We prove a version of the
Titchmarsh convolution theorem
for distributions on the circle.
We show that the ``na\"ive form''
of the Titchmarsh theorem could be violated,
but that such a violation
is only possible for the convolution
of distributions
which both possess certain symmetry properties.

\bigskip


\smallskip

(Convolution, factorization)
\end{abstract}


\section{Introduction}

The Titchmarsh convolution theorem \cite{titchmarsh}
states that
for any two compactly supported distributions
$f,\,g\in\mathscr{E}'(\R)$,
\begin{equation}\label{titchmarsh-theorem}
\inf\supp f\ast g
=\inf\supp f+\inf\supp g,
\qquad
\sup\supp f\ast g
=\sup\supp f+\sup\supp g.
\end{equation}
The higher-dimensional reformulation
by Lions
\cite{MR0043254}
states
that
for $f,\,g\in\mathscr{E}'(\R^n)$,
the convex hull of the support
of $f\ast g$
is equal to the sum of convex hulls of
supports of $f$ and $g$.
Different proofs of the Titchmarsh convolution theorem
are contained in
\cite[Chapter VI]{MR617913} (Real Analysis style),
\cite[Theorem 4.3.3]{MR1065136} (Harmonic Analysis style),
and \cite[Lecture 16, Theorem 5]{MR1400006} (Complex Analysis style).

In this note,
we generalize the Titchmarsh Theorem
to periodic distributions,
which we consider as
distributions on the circle,
or, more precisely, on the torus $\T:=\R\slash 2\pi\Z$.

First,
we note that there are zero divisors with respect
to the convolution on a circle.
Indeed, for any two distributions $f$, $g\in\mathscr{E}'(\T)$
one has
\begin{equation}\label{nt}
(f+S\sb{\pi}f)\ast(g-S\sb{\pi}g)
=f\ast g+S\sb{\pi}(f\ast g)-S\sb{\pi}(f\ast g)
-f\ast g=0.
\end{equation}
Above, $S\sb{y}$, $y\in\T$,
is the shift operator, defined on $\mathscr{E}'(\T)$ by
\begin{equation}\label{def-shift}
\big(S\sb{y}f\big)(\omega)=f(\omega-y),
\end{equation}
where the above relation
is understood
in the sense of distributions.
Yet, the cases when the Titchmarsh convolution theorem
``does not hold'' (in some na\"ive form)
could be specified.
This leads to a version of the Titchmarsh convolution theorem
for distributions on a circle
(Theorem~\ref{theorem-titchmarsh-circle} below).

\medskip





Our interest
in properties of a convolution on a circle
is due to applications
to the theory of attractors
for finite difference approximations
of nonlinear dispersive equations.
In
\cite{ubk-arma},
we considered the weak attractor
of finite energy solutions
to the $\mathbf{U}(1)$-invariant Klein-Gordon equation
in 1D,
coupled to a nonlinear oscillator.
We proved that
the global attractor
of all finite energy solutions
is formed by the set of all solitary waves,
$\phi\sb\omega(x) e^{-i\omega t}$
with
$\omega\in\R$ and $\phi\sb\omega\in H\sp 1(\R)$.
The general strategy
of the proof was to consider
the omega-limit trajectories of the finite energy solution $\psi(x,t)\in\C$,
defined as solutions
with the Cauchy data
at the omega-limit points of the set
$\{(\psi(t),\dot\psi(t))\sothat t\ge 0\}$
in the local energy seminorms.
One shows that the time spectrum
of each omega-limit trajectory
is inside the spectral gap
and then,
applying the Titchmarsh convolution theorem
to the equation satisfied by the omega-limit trajectory,
one concludes that
its time spectrum
consists of at most a single frequency,
hence any omega-limit trajectory
is a solitary wave (or zero).
For the Klein-Gordon equation
in discrete space-time \cite{MR0503140},
this approach was adapted in \cite{kg-discrete}.
The main difference is that
now the frequency domain is a circle
(no longer the whole real line)
and there are not one, but two spectral gaps in the continuous
spectrum.
Thus,
to analyze the time spectrum
of the omega-limit trajectory,
one needs a version of the Titchmarsh convolution
theorem for distributions supported inside two intervals of the circle.



\bigskip





ACKNOWLEDGMENTS.
The authors are grateful to E.A. Gorin
for his interest and support.

\section{Main results}

For $I\subset\T$
and
$n\in\N$,
denote
\[
\mathscr{R}_n(I)
=\bigcup\limits\sb{k\in\Z\sb n}S\sb{\frac{2\pi k}{n}}I,
\qquad
\mbox{where}
\quad
\Z\sb n=\Z\mod n.
\]
Let $f,\,g\in\mathscr{E}'(\T)$.
Let
$I,\,J\subset\T$ be two closed intervals 
such that
$
\supp f\subset \mathscr{R}_n(I),
$
$
\supp g\subset \mathscr{R}_n(J),
$
and assume that
there is no closed interval
$I'\subsetneq I$
such that
$\supp f\subset \mathscr{R}_n(I')$
and no closed interval $J'\subsetneq J$
such that
$\supp g\subset \mathscr{R}_n(J')$.

\begin{remark}
For $f,\,g\in\mathscr{E}'(\T)$,
the intervals $I$ and $J$
play the role similar to
``convex hulls''
of supports.
\end{remark}

\begin{theorem}[Titchmarsh theorem for distributions on a circle]
\label{theorem-titchmarsh-circle}
Let $n\in\N$, $n\ge 2$.
Assume that
\begin{equation}\label{i-j-small}
\abs{I}+\abs{J}<\frac{2\pi}{n}.
\end{equation}
Let $K\subset I+J\subset\T$
be a closed interval such that
$
\supp f\ast g\subset \mathscr{R}_n(K).
$
If $\lambda:=\inf K-\inf I-\inf J>0$,
then there are
$\alpha,\,\beta\in\C$, $\alpha^n=\beta^n=1$,
$\alpha\ne\beta$,
such that
\begin{equation}\label{asbs1}
\ \quad
\Big(\sum\sb{k\in\Z\sb n}\alpha^k S\sb{\frac{2\pi k}{n}}f\Big)
\At{(\sup I-\frac{2\pi}{n},\inf I+\lambda)}=0,
\qquad
\inf\supp
\Big(\sum\sb{k\in\Z\sb n}\alpha^k S\sb{\frac{2\pi k}{n}}g\Big)
\At{(\sup J-\frac{2\pi}{n},\inf J+\lambda)}=\inf J,
\end{equation}
\begin{equation}\label{asbs2}
\hskip -1cm
\inf\supp
\Big(\sum\sb{k\in\Z\sb n}\beta^k S\sb{\frac{2\pi k}{n}}f\Big)
\At{(\sup I-\frac{2\pi}{n},\inf I+\lambda)}=\inf I,
\qquad\quad
\Big(\sum\sb{k\in\Z\sb n}\beta^k S\sb{\frac{2\pi k}{n}}g\Big)
\At{(\sup J-\frac{2\pi}{n},\inf J+\lambda)}=0.
\end{equation}
\end{theorem}

\begin{remark}\label{remark-th1}
The relations \eqref{asbs2} follow
from \eqref{asbs1}
due to the symmetric role of $f$ and $g$.
The conclusion $\alpha\ne\beta$
follows from
comparing \eqref{asbs1} and \eqref{asbs2}.
Indeed, the first relation in \eqref{asbs1} implies that
$\inf\supp\Big(\sum\sb{k\in\Z\sb n}\alpha^k
S\sb{\frac{2\pi k}{n}}f\Big)\at{I}\ge\inf I+\lambda>\inf I$,
which would contradict
the first relation in \eqref{asbs2}
if we had $\alpha=\beta$.
\end{remark}


Applying the reflection to $\T$, we also
get the following result:

\begin{corollary}\label{corollary-easy}
If $\rho:=\sup I+\sup J-\sup K>0$,
then there are
$\alpha,\,\beta\in\C$,
$\alpha^n=\beta^n=1$,
$\alpha\ne\beta$,
such that
\begin{equation}\label{asbs3}
\Big(\sum\sb{k\in\Z\sb n}\alpha^k S\sb{\frac{2\pi k}{n}}f\Big)
\At{(\sup I-\rho,\inf I+\frac{2\pi}{n})}=0,
\qquad\quad
\sup\supp\Big(\sum\sb{k\in\Z\sb n}\alpha^k S\sb{\frac{2\pi k}{n}}g\Big)
\At{(\sup J-\rho,\inf J+\frac{2\pi}{n})}=\sup J,
\end{equation}
\begin{equation}\label{asbs4}
\hskip -1cm
\sup\supp\Big(\sum\sb{k\in\Z\sb n}\beta^k S\sb{\frac{2\pi k}{n}}f\Big)
\At{(\sup I-\rho,\inf I+\frac{2\pi}{n})}=\sup I,
\qquad
\Big(\sum\sb{k\in\Z\sb n}\beta^k S\sb{\frac{2\pi k}{n}}g\Big)
\At{(\sup J-\rho,\inf J+\frac{2\pi}{n})}=0.
\end{equation}
\end{corollary}

That is, if $K\subsetneq I+J$
(informally, we could say that
certain na\"ive form of the Titchmarsh convolution theorem is not satisfied),
then 
both $f$ and $g$ satisfy certain
symmetry properties
on $\mathscr{R}_n(U)$ and on $\mathscr{R}_n(V)$,
where open non-intersecting
intervals $U$ and $V$
can be chosen so that
$U\cup K\cup V\supset I+J$.


In the case $n=2$,
we have the following result.

\begin{corollary}\label{corollary-n2}
Let $n=2$,
$f$, $g\in\mathscr{E}'(\T)$,
and let
$I$, $J$, $K$
be as in Theorem~\ref{theorem-titchmarsh-circle}.
Then
$\lambda:=\inf K-\inf I-\inf J>0$
if and only if
there is $\alpha=\pm 1$
such that
\[
(f+\alpha S\sb\pi f)\at{(\sup I-\pi,\inf I+\lambda)}=0,
\qquad
(g-\alpha S\sb\pi g)\at{(\sup J-\pi,\inf J+\lambda)}=0.
\]
\end{corollary}

\begin{proof}[Proof of Corollary~\ref{corollary-n2}]
The ``only if'' part follows
from Theorem~\ref{theorem-titchmarsh-circle}.
We check the ``if'' part
by direct computation.
Let
$f\in\mathscr{E}'(I\cup S\sb{\pi}I)$,
where $I\subset\T$, $\abs{I}<\pi/2$,
$g\in\mathscr{E}'(J\cup S\sb{\pi}J)$,
where $J\subset\T$, $\abs{J}<\pi/2$,
and assume that
$f=\pm S\sb\pi f$ on $(\sup I-\pi,\inf I+\lambda)$,
$g=\mp S\sb\pi g$ on $(\sup J-\pi,\inf J+\lambda)$.
Then, as in \eqref{nt},
\begin{eqnarray}
&&
(f\ast g)\at{(\sup I+\sup J-2\pi,\inf I+\inf J+\lambda)}
\nonumber
\\
&&
=f\at{(\sup I-\pi,\inf I+\lambda)}
\ast g\at{(\sup J-\pi,\inf J+\lambda)}
+(S\sb\pi f)\at{(\sup I-\pi,\inf I+\lambda)}
\ast(S\sb\pi g)\at{(\sup J-\pi,\inf J+\lambda)}
\nonumber
\\
&&
=f\at{(\sup I-\pi,\inf I+\lambda)}
\ast g\at{(\sup J-\pi,\inf J+\lambda)}
-f\at{(\sup I-\pi,\inf I+\lambda)}
\ast g\at{(\sup J-\pi,\inf J+\lambda)}=0.
\nonumber
\end{eqnarray}
\end{proof}

Define
$f\sp\sharp(\omega)=f(-\omega)$.
Let $f\in\mathscr{E}'(\T)$
and let
$I\subset\T$
be a closed interval such that
$\supp f\subset \mathscr{R}_2(I)$.
Assume that
there is no
closed interval $I'\subsetneq I$ such that
$\supp f\subset \mathscr{R}_2(I')$.

\begin{theorem}\label{theorem-titchmarsh-weird}
If
$I\subset(-\pi/2,\pi/2)$
and
$\abs{I}<\pi/2$,
then
the inclusion
$
\supp f\ast f\sp\sharp\subset\{0;\pi\}
$
implies that
$
\supp f\subset\{\inf I;\sup I;\pi+\inf I;\pi+\sup I\}.
$
Moreover, there are
distributions $\mu$, $\nu\in\mathscr{E}'(\T)$,
each supported at a point,
such that
\begin{equation}\label{fab1}
f=\mu+S\sb{\pi}\mu+\nu-S\sb{\pi}\nu.
\end{equation}
\end{theorem}

\begin{remark}
The statement of Theorem~\ref{theorem-titchmarsh-weird}
remains true if one defines
$f\sp\sharp(\omega)=\overline{f(-\omega)}$
(the form used in \cite{kg-discrete}).
This change does not affect the proof.
\end{remark}

Finally, let us also formulate the
convolution theorem for powers of a distribution.
Let $f\in\mathscr{E}'(\T)$.
Let $I\subset\T$ be a closed interval
such that $\supp f\subset \mathscr{R}_n(I)$,
and assume that there is no $I'\subsetneq I$
such that
$\supp f\subset \mathscr{R}_n(I')$.

\begin{theorem}[Titchmarsh theorem for powers
of a distribution on a circle]
\label{theorem-titchmarsh-powers}
Assume that $\abs{I}<\frac{2\pi}{pn}$, for some $p\in\N$.
Then the smallest closed interval
$K\subset p I$ such that
$\supp f\sp{\ast p}\subset \mathscr{R}_n(K)$
is $K=p I$.
\end{theorem}

Above, we used the notations
$pI=\underbrace{I+\dots+I}\sb{p}$
and
$f\sp{\ast p}=\underbrace{f\ast\dots\ast f}\sb{p}$.

\section{Proofs}

First, we prove the following lemma.

\begin{lemma}\label{lemma-ssj}
Let $f\sb j\in\mathscr{E}'(I)$,
$j\in\Z\sb n$.
There is $\alpha\in\C$, $\alpha^n=1$,
such that
\begin{equation}\label{in-lemma-ssj}
\inf\supp\sum\sb{j\in\Z\sb n}\alpha^j f\sb j
=\min\limits\sb{j\in\Z\sb n}
\inf\supp f\sb j.
\end{equation}
\end{lemma}

\begin{proof}
Denote
$a:=\min\limits\sb{j\in\Z\sb n}\inf\supp f\sb j$.
Let us assume that,
contrary to the statement of the lemma,
there is $\epsilon>0$
such that
$
\inf\supp\sum\sb{j\in\Z\sb n}\alpha^j f\sb j\ge a+\epsilon,
$
for any $\alpha=\gamma^m$,
where $\gamma=\exp(\frac{2\pi i}{n})$
and $m\in\N$, $1\le m\le n$.
Then
for any test function $\varphi\in\mathscr{D}(\R)$
with $\supp\varphi\subset(a-\epsilon,a+\epsilon)$
we would have:
\begin{equation}\label{phi-gamma-f}
0=\langle\varphi,\sum\sb{j\in\Z\sb n}\gamma^{jm} f\sb j\rangle
=\sum\sb{j\in\Z\sb n}\gamma^{jm}\langle\varphi,f\sb j\rangle,
\qquad
1\le m\le n.
\end{equation}
Using the formula for the Vandermonde determinant,
we have
\begin{equation}\label{vdmd}
\det
\left[
\begin{matrix}
1&\gamma&\gamma^2&\cdots&\gamma^{n-1}
\\
1&\gamma^2&\gamma^4&\cdots&\gamma^{2(n-1)}
\\
1&\gamma^3&\gamma^6&\cdots&\gamma^{3(n-1)}
\\
\vdots&\vdots&\vdots&\ddots&\vdots
\\
1&\gamma^n&\gamma^{2n}&\cdots&\gamma^{n(n-1)}
\end{matrix}
\right]
=\prod\sb{1\le j<k\le n}(\gamma^k-\gamma^j)\ne 0.
\end{equation}
Hence, \eqref{phi-gamma-f} implies that
$\langle\varphi,f\sb j\rangle=0$ for all $j\in\Z\sb n$.
Due to arbitrariness of $\varphi$,
this leads to
$f\sb j\at{(a-\epsilon,a+\epsilon)}=0$
for all $j\in\Z\sb n$,
leading to a contradiction
with the definition of $a$.
\end{proof}

\begin{proof}[Proof of Theorem~\ref{theorem-titchmarsh-circle}]
One has
$\supp f\subset \mathscr{R}_n(I)$,
$\supp g\subset \mathscr{R}_n(J)$,
$\supp f\ast g\subset\mathscr{R}_n(K)\subset\mathscr{R}_n(I+J)$.
Due to the restriction
\eqref{i-j-small},
the sets
$\mathscr{R}_n(I)$, $\mathscr{R}_n(J)$,
and $\mathscr{R}_n(I+J)$
each consist of $n$ non-intersecting intervals.
For $j\in\Z\sb n$,
let us set
$f\sb j=(S\sb{\frac{2\pi j}{n}}f)\at{I}\in\mathscr{E}'(I)$,
$g\sb j=(S\sb{\frac{2\pi j}{n}}g)\at{J}\in\mathscr{E}'(J)$,
$h\sb j=\big(S\sb{\frac{2\pi j}{n}}(f\ast g)\big)\at{K}\in\mathscr{E}'(I+J)$;
then
\begin{equation}\label{hfg}
h\sb j
=
\big(S\sb{\frac{2\pi j}{n}}(f\ast g)\big)\At{I+J}
=
\sum\sb{\stackrel{\scriptstyle k+l=j\!\!\mod n}{\scriptstyle k,\,l\in\Z\sb n}}
(S\sb{\frac{2\pi k}{n}}f)\at{I}
\ast(S\sb{\frac{2\pi l}{n}}g)\at{J}
=
\sum\sb{\stackrel{\scriptstyle k+l=j\!\!\mod n}{\scriptstyle k,\,l\in\Z\sb n}}
f\sb k\ast g\sb l,
\qquad
j\in\Z\sb n.
\end{equation}
Using the relation \eqref{hfg},
for any $\alpha\in\C$ such that $\alpha^n=1$
we have:
\begin{equation}\label{wh}
\Big(
\sum\sb{k\in\Z\sb n}\alpha^k f\sb k
\Big)
\ast
\Big(\sum\sb{l\in\Z\sb n}\alpha^l g\sb l
\Big)
=\sum\sb{j\in\Z\sb n}
\alpha^j
\Big[
\sum\sb{\stackrel{\scriptstyle k+l=j\!\!\mod n}{\scriptstyle k,\,l\in\Z\sb n}}
f\sb k\ast g\sb l
\Big]
=
\sum\sb{j\in\Z\sb n}\alpha^j h\sb j.
\end{equation}
Applying the Titchmarsh convolution theorem \eqref{titchmarsh-theorem}
to this relation,
we obtain:
\begin{equation}\label{isf-isg}
\inf\supp
\sum\sb{j\in\Z\sb n}\alpha^j f\sb j
+
\inf\supp
\sum\sb{j\in\Z\sb n}\alpha^j g\sb j
=
\inf\supp
\sum\sb{j\in\Z\sb n}\alpha^j h\sb j
\ge\inf K,
\end{equation}
where we took into account that
$\min\sb{j\in\Z\sb n}\inf\supp h\sb j\ge\inf K$.
By Lemma~\ref{lemma-ssj},
there is $\alpha\in\C$, $\alpha^n=1$,
such that
$
\inf\supp\sum\sb{j\in\Z\sb n}\alpha^j g\sb j
=
\min\limits\sb{j\in\Z\sb n}
\inf\supp g\sb j
=\inf J$;
this is equivalent to the second relation in \eqref{asbs1}.
For this value of $\alpha$,
\eqref{isf-isg} yields:
\[
\inf\supp\sum\sb{j\in\Z\sb n}\alpha^j f\sb j
\ge\inf K-\inf J=\inf I+\lambda.
\]
This is equivalent to the first relation in \eqref{asbs1}.
According to Remark~\ref{remark-th1},
this finishes the proof.
\end{proof}

\begin{proof}[Proof of Theorem~\ref{theorem-titchmarsh-weird}]
If $I$ consists of one point,
$I=\{p\}\subset(-\pi/2,\pi/2)$,
then
$\supp f=\mathscr{R}\sb 2(p)=\{p;\pi+p\}$,
and \eqref{fab1} holds with
\[
\mu=\frac{f+S\sb\pi f}{2}\At{I},
\quad
\nu=\frac{f-S\sb\pi f}{2}\At{I}.
\]
Now we assume that
$\abs{I}>0$.
Define $J=-I$ and $K=\{0\}\subset I+J$.
Then
$\supp f\sp\sharp\subset\mathscr{R}\sb 2(J)$
and there is no $J'\subsetneq J$
such that
$\supp f\sp\sharp\subset\mathscr{R}\sb 2(J')$.
According to the conditions of the theorem,
$\supp f\ast f\sp\sharp\subset\mathscr{R}\sb 2(K)$;
hence, one has:
\begin{equation}
\lambda:=\inf K-\inf I-\inf J=\sup I-\inf I=\abs{I}>0.
\label{ff1}
\end{equation}
Applying Theorem~\ref{theorem-titchmarsh-circle}
to \eqref{ff1}, we conclude that
there is $\alpha\in\{\pm 1\}$
such that
\begin{equation}\label{asdf1}
(f+\alpha S\sb{\pi}f)\at{(\sup I-\pi,\sup I)}=0
\end{equation}
and also
$\inf\supp(f\sp\sharp+\alpha S\sb{\pi}f\sp\sharp)\at{(-\frac{\pi}{2},\frac{\pi}{2})}
=
-\sup I$;
this last relation implies that
\begin{equation}\label{asdf1a}
\sup\supp(f+\alpha S\sb{\pi}f)\at{(-\frac{\pi}{2},\frac{\pi}{2})}
=
\sup I.
\end{equation}
Similarly, by Theorem~\ref{theorem-titchmarsh-circle},
there is $\beta\in\{\pm 1\}$
such that
$
(f\sp\sharp+\beta S\sb{\pi}f\sp\sharp)\at{(-\inf I-\pi,-\inf I)}=0,
$
hence
\begin{equation}\label{asdf3a}
(f+\beta S\sb{\pi}f)\at{(\inf I,\inf I+\pi)}=0.
\end{equation}
Comparing \eqref{asdf1a} with \eqref{asdf3a},
we conclude that $\alpha\ne\beta$, hence $\alpha=-\beta$;
then \eqref{asdf1} and \eqref{asdf3a}
allow us to conclude that both $f$ and $S\sb\pi f$
vanish on $(\inf I,\sup I)$,
hence
\[
\supp f\subset\{\inf I;\sup I;\pi+\inf I;\pi+\sup I\}.
\]
By \eqref{asdf1} and \eqref{asdf3a},
if $\alpha=1$,
the relation \eqref{fab1} holds
with $\mu=f\at{(\inf I,\pi/2)}$
and $\nu=f\at{(-\pi/2,\sup I)}$.
If instead $\alpha=-1$,
the relation \eqref{fab1} holds
with $\mu=f\at{(-\pi/2,\sup I)}$
and $\nu=f\at{(\inf I,\pi/2)}$.
\end{proof}

Let us notice that the proof
of Theorem~\ref{theorem-titchmarsh-powers}
for the case $p=2$
immediately follows from Theorem~\ref{theorem-titchmarsh-circle}.
(For example, the relations \eqref{asbs1} with $f=g$
are mutually contradictory unless $\lambda=0$.)
By induction,
this also gives the proof for $p=2^N$, with any $N\in\N$.
Then one can deduce the statement of
Theorem~\ref{theorem-titchmarsh-powers}
for any $p\le 2^N$, but under the condition
$\abs{I}<\frac{2\pi}{2^N n}$,
which is stronger than
$\abs{I}<\frac{2\pi}{p n}$.
Instead of trying to use Theorem~\ref{theorem-titchmarsh-circle},
we give an independent proof.

\begin{proof}[Proof of Theorem~\ref{theorem-titchmarsh-powers}]

One has
$\supp f\sp{\ast p}\subset\mathscr{R}_n(p I)$.
Due to the smallness of $I$,
both $\mathscr{R}_n(I)$ and $\mathscr{R}_n(p I)$
are collections of $n$ non-intersecting
intervals.
Define $f\sb j:=(S\sb{\frac{2\pi j}{n}}f)\at{I}\in\mathscr{E}'(I)$
and
$h\sb j:=
\big(S\sb{\frac{2\pi j}{n}}(f\sp{\ast p})\big)
\at{I}\in\mathscr{E}'(I)$.
Then
\begin{equation}\label{hff}
h\sb j
=\big(S\sb{\frac{2\pi j}{n}}(f\sp{\ast p})\big)\At{pI}
=
\!\!\!\!
\sum\sb{
\stackrel
{\scriptstyle j\sb 1+\dots+j\sb p=j\!\!\mod n}
{\scriptstyle j\sb 1,\,\dots,\,j\sb p\in\Z\sb n}}
\!\!\!\!
(S\sb{\frac{2\pi j\sb 1}{n}}f)\at{I}\ast
\dots\ast(S\sb{\frac{2\pi j\sb p}{n}}f)\at{I}
=
\!\!\!\!
\sum\sb{
\stackrel
{\scriptstyle j\sb 1+\dots+j\sb p=j\!\!\mod n}
{\scriptstyle j\sb 1,\,\dots,\,j\sb p\in\Z\sb n}}
\!\!\!\!
f\sb{j\sb 1}\ast\dots\ast f\sb{j\sb p},
\qquad
j\in\Z\sb n.
\end{equation}
Taking into account \eqref{hff},
for any $\alpha\in\C$ such that $\alpha^n=1$
one has:
\begin{equation}\label{ffhj}
\Big(\sum\sb{j\in\Z\sb n}\alpha^j f\sb j\Big)\sp{\ast p}
=
\sum\sb{j\in\Z\sb n}
\alpha^j
\Big[
\sum\sb{j\sb 1+\dots+j\sb p=j\!\!\mod n}
f\sb{j\sb 1}\ast\dots\ast f\sb{j\sb n}
\Big]
=
\sum\sb{j\in\Z\sb n}
\alpha^j
h\sb j.
\end{equation}
Now we apply the Titchmarsh convolution theorem
to \eqref{ffhj}, getting
\[
p\inf\supp\sum\sb{j\in\Z\sb n}\alpha^j f\sb j
=\inf\supp\sum\sb{j\in\Z\sb n}\alpha^j h\sb j.
\]
By Lemma~\ref{lemma-ssj},
there is $\alpha\in\C$, $\alpha^n=1$, such that
$
\inf\supp\sum\sb{j\in\Z\sb n}\alpha^j f\sb j
=\min\limits\sb{j\in\Z\sb n}
\inf\supp f\sb j$,
hence,
for this value of $\alpha$,
\[
p\min\limits\sb{j\in\Z\sb n}\inf\supp f\sb j
=\inf\supp\sum\sb{j\in\Z\sb n}\alpha^j h\sb j
\ge\min\limits\sb{j\in\Z\sb n}\inf\supp h\sb j.
\]
On the other hand,
\eqref{hff} immediately yields the inequalities
$\inf\supp h\sb j\ge p\min\limits\sb{k\in\Z\sb n}\inf\supp f\sb k$,
for any $j\in\Z\sb n$.
It follows that
$\min\limits\sb{j\in\Z\sb n}\inf\supp h\sb j=p\min\limits\sb{j\in\Z\sb n}\inf\supp f\sb j$
and similarly
$\max\limits\sb{j\in\Z\sb n}\sup\supp h\sb j=p\max\limits\sb{j\in\Z\sb n}\sup\supp f\sb j$.
\end{proof}

\hbadness=10000

\def\cprime{$'$} \def\cprime{$'$} \def\cprime{$'$} \def\cprime{$'$}
  \def\cprime{$'$} \def\cprime{$'$} \def\cprime{$'$} \def\cprime{$'$}
  \def\cprime{$'$} \def\cprime{$'$} \def\cprime{$'$}

\end{document}